\documentclass{birkjour}
\usepackage{graphicx}
\usepackage{amssymb,amsmath}
\usepackage{amsthm}
\usepackage{subfigure}
\usepackage{float}
\usepackage{epsfig}
\usepackage{rotating}
\usepackage{MnSymbol}
\usepackage{qtree}
\usepackage{amsmath}
\usepackage{mathtools}
\newcommand{\U}{\mathbin{\mathcal{U}\kern-.1em}}
\renewcommand{\S}{\mathbin{\mathcal{S}\kern-.08em}}

\renewcommand{\a}{\alpha}
\renewcommand{\b}{\beta}

\renewcommand{\O}{\Omega}

\newcommand{\mc}{\mathcal}

\theoremstyle{plain}
\newtheorem{proposition}{Proposition}[section]
\newtheorem{theorem}[proposition]{Theorem}

\theoremstyle{definition}
\newtheorem{definition}[proposition]{Definition}

 \theoremstyle{definition}
 
 \theoremstyle{remark}

 \numberwithin{equation}{section}

\begin{document}
\title[]
 {Hyperbolic solitons on trans-Sasakian space forms and its submanifolds}

\author{Bidhan Mondal}

\address{%
Department of Mathematics, Jadavpur University}

\email{bidhanmondal381@gmail.com}

\thanks{The author Bidhan Mondal thanks the UGC JRF  (Id No. 231610077944) for their financial assistance.}
\author{Sibsankar Panda}
\address{Department of Mathematics, Sidho-Kanho-Birsha University }
\email{shibu.panda@gmail.com}
\author{Nirabhra Basu}

\address{%
Department of Mathematics, Bhawanipur Education Society College}

\email{basu.nirabhra@gmail.com}




\dedicatory{}

\begin{abstract}
    Kong and Liu introduced the concept of hyperbolic Ricci flow in 2007 and used it to study the wave character of metrics. After that, many mathematicians have used this new geometric flow to study the evolution of manifolds and their structures. Ricci solitons and hyperbolic Ricci solitons are self-similar solitons of the Ricci flow and hyperbolic Ricci flow respectively.\\
    In this paper, we introduce the concept of hyperbolic $*-$Ricci solitons and hyperbolic Ricci-Yamabe solitons on a trans-Sasakian space forms and characterized the nature of some hyperbolic solitons. Additionally, we deduce the Ricci tensors of submanifolds of trans-Sasakian space forms and conformal trans-Sasakian space form and found the nature of solitons on submanifolds. Finally, we have included an example which will justify our result.
\end{abstract}

\maketitle
\section{Introduction}
In differential geometry, geometric flows play a crucial role since they help us to study canonical metrics on the underlying Riemannian manifolds. A  geometric flow is the progression of a geometric configuration under a manifold-specific differential equation with a functional. In order to study compact three-dimensional manifolds with positive Ricci curvature,  Hamilton introduced one of the most significant geometric flows, the Ricci flow \cite{CBP}. In recent years, there has been a significant increase in the study of various types of solitons as stationary solutions of various geometric flows. The main aim is to determine the geometrical properties of a soliton, particularly those related to its curvature and to establishing obstacles for a manifold to be a soliton. Hyperbolic geometric flow is one of these geometric flows; it is a system of second-order nonlinear evolution partial differential equations and is very close to wave equation flow metrics. It was first introduced by Kong and Liu \cite{KL} in 2007 to characterize the wave phenomenon of metrics and  curvatures of Riemannian manifolds. The general version of hyperbolic geometric flow for the metric $g$ is given by 
\begin{equation}\label{eq001}
    \frac{\partial^2 g}{\partial t^2}+2S+\mc{F}\left(g,\frac{\partial g}{\partial t}\right)=0, \quad g(0) = g_0,
\end{equation}
where $S$ is the Ricci tensor and $\mc{F}$ is a given smooth function of the Riemannian metric
$g$ and its first-order derivative with respect to $t$. The three most important special cases are the so-called standard hyperbolic geometric flow, or simply called the hyperbolic geometric flow, Einstein’s hyperbolic geometric flow and the dissipative hyperbolic geometric flow. In the present paper, we consider the
standard hyperbolic geometric flow on a Riemannian manifold. Hyperbolic Ricci flow is defined by 
\begin{equation}\label{eq002}
    \frac{\partial^2 g}{\partial t^2}+2S=0,\quad g(0) = g_0,\quad \frac{\partial g}{\partial t} (0) = k_0,
\end{equation}
where $k_0$ is a symmetric 2-tensor field on the Riemannian manifold.
Recently, in 2023, the authors Hamed Faraji, Shahroud Azami and Ghodratallah Fasihi-Ramandi \cite{HSG} introduced the hyperbolic Ricci soliton which is a self-similar solution of the hyperbolic Ricci flow. A Riemannian manifold $(M, g)$ is said to admit a hyperbolic Ricci soliton if there exists a vector field $V$ on $M$ and real scalars $\mu$ and $\lambda$ satisfying 
     \begin{equation}\label{eq004}
         \mc{L}_V(\mc{L}_Vg)+2\lambda\mc{L}_Vg+2S=2\mu g,
     \end{equation}
where $\mc{L}$ is the Lie derivative along the vector field $V$. The soliton is expanding, steady, or shrinking according as $\lambda>0$, $ \lambda = 0$ or $ \lambda< 0$, respectively. If $V = 0$ or $V$ is a Killing vector field, then (\ref{eq004}) reduces to an Einstein manifold. A hyperbolic Ricci soliton $(M, g, V,\mu,\lambda)$ is called gradient hyperbolic Ricci soliton if there exists a function $f$ (called a
potential function) such that $V = \nabla f$. In this case, (\ref{eq004}) can be rewritten as
\begin{equation}
    \mc{L}_ {\nabla f}(\nabla^2 f)+2\lambda\nabla^2 f+S=2\mu g.
\end{equation}

Tachibana \cite{TS} and Hamada \cite{TH} introduced the notion of $*$-Ricci tensor on almost Hermitian manifolds and on real hypersurfaces in non-flat complex space, respectively. Subsequently, in 2014, Kaimakamis and Panagiotidou \cite{GK} introduced the notion of $*-$Ricci soliton on non-flat complex space forms, satisfying the equation
\begin{equation} 
\mc{L}_V g+2S^*+2\lambda g=0,
\end{equation}
where $S^*(X,Y)=-\frac{1}{2}\left[ {trace}\{\varphi\circ R(X,\varphi Y)\}\right] $ for all vector fields $X, Y$ on $M$ and $\varphi$ is a (1,1)-tensor field.

In 2015, N. Basu and A. Bhattacharyya  introduced conformal Ricci soliton \cite{BNA} and recently proved its existence  \cite{NAB}. The conformal Ricci soliton equation was given by 
$$\mc{L}_Vg+2S=2[\lambda -\frac{1}{2}(p+\frac{2}{n})] g.$$

Hamilton \cite{HH} introduced the idea of Yamabe flow on compact Riemannian manifolds. A Yamabe soliton \cite{EE} corresponds to a self-similar solution of the Yamabe flow, defined on a Riemannian or pseudo-Riemannian manifold $(M, g)$ by a vector field $V$ satisfying the equation,
\begin{equation}
    \mc{L}_Vg=2(r-\lambda)g,
\end{equation}
where $r$ is the scalar curvature and $\lambda$ is a constant. Moreover a Yamabe soliton is said to be expanding if $\lambda<0$, steady if $\lambda=0$ and shrinking if $\lambda>0$.
In the recent year, A. M. Blaga \cite{BO, BAM} has defined the hyperbolic Yamabe flow and studied some properties of hyperbolic Yamabe solitons.  A Riemannian manifold $(M, g)$ is said to admit hyperbolic Yamabe soliton if there exists a vector field $V$ and a constant $\lambda $ such that
     \begin{equation}\label{eq007}
        \mc{L}_V(\mc{L}_V g)+\lambda\mc{L}_V g=(\mu-r) g,
     \end{equation}
    where $\gamma$ and $\mu$ are real scalars.

These recent studies encourage us to introduce a new flow and its soliton, namely hyperbolic Ricci-Yamabe flow and soliton. Next we have brought the $*-$Ricci curvature tensor in a trans-Sasakian space form. After that we investigate the nature of some hyperbolic solitons on trans-Sasakian space form and conformal trans-Sasakian space form. Finally we present an example. 

\begin{definition}
     A Riemannian manifold $(M, g)$ is said to admit a hyperbolic $*-$Ricci soliton if there exists a vector field $V$ on $M$ and real scalars $\mu$ and $\lambda$ such that
     \begin{equation}\label{eq005}
         \mc{L}_V(\mc{L}_Vg)+2\lambda\mc{L}_Vg+2S^*=2\mu g.
     \end{equation}
\end{definition}
The soliton is expanding, steady, or shrinking if  $\lambda$ is positive, $\lambda = 0$ or $ \lambda< 0$, respectively. Moreover, $M$ is said to be a $*-$Einstein manifold if $S^*(X,Y)=\lambda g(X,Y)$, where $\lambda$ is a constant and a $*$$-\eta-$Einstein manifold if  $S^*(X,Y)=\lambda g(X,Y)+\mu\eta(X)\eta(Y)$ where $\mu$ is real.

\begin{definition}\label{de01}\cite{BN}
     An n-dimensional Riemannian manifold $(M,g)$ is said to admit hyperbolic conformal Ricci soliton if there exists a vector field $V$ on $M$ and real scalars $\mu$ and $\lambda$ such that
     \begin{equation}\label{eq006}
         \mc{L}_V(\mc{L}_Vg)+2\lambda\mc{L}_Vg+2S=2[\mu-\frac{1}{2}(p+\frac{2}{n})] g.
     \end{equation}
\end{definition}
The soliton is expanding, steady or shrinking if  $\lambda$ is positive, $ \lambda = 0$ or $ \lambda$ is negative, respectively.

\section{Preliminaries}
A three-dimensional manifold $M$ is said to be an almost contact manifold if a tensor field $\varphi$ of type $(1,1)$, a vector field $\xi$ and a $1-$form $\eta$ on $M$ satisfies the following conditions: 
\begin{equation}\label{eqb01}
\eta(\xi)=1,\quad \varphi^2=-I+\eta\otimes\xi,
\end{equation}
for any vector field $X$ on $M$.
From above equation, it is easy to prove that,
$$\varphi(\xi)=0,\ \eta\circ\varphi=0.$$

Every almost contact manifold contains a Riemannian metric $g$, which satisfies the conditions:
\begin{equation}\label{eqb02}
    g(X,\xi)=\eta(X),\quad g(\xi,\xi)=1,
\end{equation}
\begin{equation}\label{eqb03}
    g(\varphi X,\varphi Y)=g(X,Y)-\eta(X)\eta(Y),
\end{equation}
\begin{equation}\label{eqb04}
    g(\varphi X,Y)=-g(X,\varphi Y).
\end{equation}
This structure $(M,\varphi,\xi,\eta,g)$ is called an almost contact metric manifold. 
 An almost contact metric manifold $(M,\varphi,\xi,\eta,g,\a,\b)$ is said to be a trans-Sasakian manifold if and only if it satisfies the following condition \cite{DMM}
\begin{equation}\label{eqb05}
     (\nabla_X\varphi )Y=\a[g(X,Y)\xi-\eta(Y)X]+\b[g(\varphi X,Y)\xi-\eta(Y)\varphi X],
\end{equation}
for all $X,Y\in TM$ and $\a$ and $\b$ are smooth functions.
The trans-Sasakian manifold of type $(0,0), (\a,0)$ and $ (0,\b)$ are called cosymplectic, $\a$-Sasakian manifold and $\b$-Kenmotsu manifold respectively. Sasakian manifold is a particular case of $\a$-Sasakian manifold. $\b-$Kenmotsu manifold becomes Kenmotsu manifold when $\a=0$ and $\b=1$. In 1992, Marrero \cite{MJc} has shown that a trans-Sasakian manifold of dimension $\geq 5$ is either cosymplectic manifold or $\a$-Sasakian manifold or $\b$-kenmotsu manifold.
Now from the above equation (\ref{eqb05}), it can be derived that
\begin{equation}\label{eqb06}
    \nabla_X\xi=-\a\varphi(X)-\b\varphi^2(X),
\end{equation}
\begin{equation}\label{eqb07}
    (\nabla_X\eta)(Y)=-\a g(X,\varphi Y)+\b(\varphi X,\varphi Y).
\end{equation}
S. Panda and A. Bhattacharya \cite{SKA} introduced the trans-Sasakian space form. A trans-Sasakian manifold $M$ with constant $\varphi$-sectional curvature $c$ is called trans-Sasakian space form and its curvature tensor is given as
\begin{align}\label{eqb08}
  R(X, Y)Z =\frac{3(\a^2-\b^2)+c}{4}[g( Y, Z)X-g(X,Z)Y]-\frac{\a^2-\b^2-c}{4}\{\eta(Z)\nonumber\\
  [\eta(X)Y-\eta(Y)X]+[\eta(Y)g(X,Z)-\eta(X)g( Y, Z)]\xi+[g(X,\varphi Z)\varphi Y\nonumber\\-g(Y,\varphi Z)\varphi X+2g(X,\varphi Y)\varphi Z]\} -2\a\b \{[\eta(Y)g(\varphi Z,X)\nonumber\\-\eta(X)g(\varphi Z,Y)]\xi+\eta(Z)[\eta(X)\varphi Y-\eta(Y)\varphi X]\}.
\end{align} 
Correspondingly, Ricci curvature tensor is
\begin{equation}\label{eqb10}
    S(\xi,Y)=2(\a^2-\b^2)\eta(Y),
\end{equation}
and scalar curvature is
\begin{equation}\label{eqb11}
    r=4(\a^2-\b^2)+2c.
\end{equation}
These results will be used in the next sections. 
\section{Hyperbolic Ricci-Yamabe Flow and Soliton}
In this section we shall see the characterization of a three dimensional trans-Sasakian space form admitting hyperbolic Ricci-Yamabe soliton and show that such a space form is $\eta-$Einstein.\\
We introduce the notion of hyperbolic Ricci-Yamabe flow as an evolution equation
\begin{align}\label{eqc01}
    \frac{\partial^2g}{\partial t^2}(t)=-2aS-brg(t),
\end{align}
where $g$ is a time-dependent Riemannian metric, $a,b$ are real numbers and $r$ stands for the scalar curvature on an $n-$dimensional smooth manifold $M$.
We take a vector field $\mc{V}=\frac{1}{f(t)}\mc{V}_0$, for $f(t)$ a smooth positive real function in an open interval $I$ containing 0 and $f(0)=1$. Now we consider a family of Riemannian metrics
$$g(t)=f(t)\psi_t^*(g_0),$$
where $t\in I$ and $\psi_t^*$ is a $1-$parameter family of diffeomorphisms function from $M$ to $M$. Then
$$\frac{\partial g}{\partial t}(t)=f'(t)\psi_t^*(g_0)+f(t)\psi_t^*(\mc{L_\xi}g_0)$$
and
\begin{align}\label{eqc02}
    \frac{\partial^2g}{\partial t^2}(t)=f''(t)\psi_t^*(g_0)+2\frac{f'(t)}{f(t)}\psi_t^*(\mc{L_V}_0g_0)+\frac{1}{f(t)}\psi_t^*(\mc{L_V}_0\mc{L_V}_0g_0),
\end{align}
where $\mc{L_V}_0g_0$ and $\mc{L_V}_0\mc{L_V}_0g_0$ are first-order and second-order Lie derivatives of the metric $g_0$. Now choosing $f$ to satisfy $f'(0)=\lambda$ and $f''(0)=-2\mu$, from (\ref{eqc01}) and (\ref{eqc02}), we get the equation of the stationary solutions of the hyperbolic Ricci-Yamabe flow, namely, the hyperbolic Ricci-Yamabe soliton
\begin{align}\label{eqc03}
    \mc{L_V}\mc{L_V}g+2\lambda\mc{L_V}g+2aS=(2\mu-br)g,
\end{align}
for $\lambda,~\mu\in\mathbb{R}$.
\begin{theorem}
    Let a three-dimensional trans-Sasakian space form $(M,\varphi,\xi,\mu,\lambda,\\\a,\b)$ with $\b\neq0$ admits an hyperbolic Ricci-Yamabe soliton, then the manifold $(M, g)$ is $\eta-$Einstein and $\mu=2a(\a^2-\b^2)+\frac{1}{2}br$, where $r$ is scalar curvature and $\lambda=\frac{a}{\b}(\a^2-\b^2-c)-2\b$.
\end{theorem}

\begin{proof}
    From the definition (\ref{eqc03}) of hyperbolic Ricci-Yamabe soliton on a manifold $(M^3,\varphi,\xi,\mu,\lambda)$,
\begin{equation}\label{eqc04}
    \mc{L_V}\mc{L_V}g(X,Y)+2\lambda\mc{L_V}g(X,Y)+2aS(X,Y)=(2\mu-br) g(X,Y),
\end{equation}
 for all $X,Y \in TM$ and $a,b$ are real number.
Replacing $\mc{V}$ with $\xi$, we obtain
\begin{equation}\label{eqc05}
    \mc{L}_\xi\mc{L}_\xi g(X,Y)+2\lambda\mc{L}_\xi g(X,Y)+2aS(X,Y)=(2\mu-br) g(X,Y).
\end{equation}
Now applying the Lie derivative property, $(\mc{L}_\xi g)(X,Y) = g(\nabla_X\xi, Y ) +\\ g(X,\nabla_Y\xi)$, and using (\ref{eqb06}), we obtain
    \begin{equation}\label{eqc06}
        (\mc{L}_\xi g)(X,Y) =2\b g(\varphi X,\varphi Y).
    \end{equation}
    Taking Lie-derivative of equation (\ref{eqc02}) and using (\ref{eqb05}) and (\ref{eqb06}), we obtain
    \begin{equation}\label{eqc07}
        (\mc{L}_\xi(\mc{L}_\xi g))(X,Y)=4\b^2g(\varphi^2 X,\varphi^2 Y).
    \end{equation}
Using (\ref{eqc06}) and (\ref{eqc07}) in equation \ref{eqc05}, we obtain,
$$\mu=2a(\a^2-\b^2)+\frac{1}{2}br,$$
and
$$\lambda=\frac{a}{\b}(\a^2-\b^2-c)-2\b.$$ 
The nature of the soliton depends on the value of $\lambda$. The hyperbolic $*-$Ricci soliton is shrinking, steady, and expanding according to whether $\lambda$ is negative, zero and positive respectively. 
So the  soliton is shrinking, steady and expanding if $c>\a^2-(1+\frac{2}{a})\b^2,~c=\a^2-(1+\frac{2}{a})\b^2~and~ c<\a^2-(1+\frac{2}{a})\b^2$.
\end{proof}

\section{Some types of hyperbolic solitons on trans-Sasakian space form}

In this section, we have studied $*-$Ricci solitons on trans-Sasakian space form and discussed the nature of $*-$Ricci solitons and hyperbolic $*-$Ricci solitons on a three-dimensional trans-Sasakian space form. Finally, we have obtained the conditions of hyperbolic Ricci soliton, hyperbolic conformal Ricci soliton and hyperbolic Yamabe soliton on a three-dimensional trans-Sasakian space form to be steady, expanding and shrinking.

\begin{theorem}
    If a three-dimensional trans-Sasakian space form $(M,g,\varphi,\eta,\lambda)$ admits a $*-$Ricci soliton, then $M$ is $*$$-\eta-$Einstein, $\lambda =\frac{1}{3}(c-2\b)$.
\end{theorem}
\begin{proof}
    Let $(M,\varphi,\xi,\eta,g,\a,\b, c)$ be a trans-Sasakian space form, where $\a,\b$ are constants and c is the $\varphi-$sectional curvature. Then by the definition of $*-$Ricci curvature tensor $S^*$ of $M$, we replace $Z$ by $\varphi$Z in equation (\ref{eqb08}) to get
      \begin{align}
          R(X,Y)\varphi Z=\frac{3(\a^2-\b^2)+c}{4}[g(Y,\varphi Z)X-g(X,\varphi Z)Y]-\frac{\a^2-\b^2-c}{4}\nonumber\\\{[\eta(Y)g(X,\varphi Z)-\eta(X)g(Y,\varphi Z)]\xi-g(\varphi X,\varphi Z)\varphi Y\nonumber\\+g(\varphi Y,\varphi Z)\varphi X+2g(X,\varphi Y)\}+2\a\b \{[\eta(Y)g(\varphi Z,\varphi X)\nonumber\\-\eta(X)g(\varphi Z,\varphi Y)]\xi\}\nonumber.
      \end{align}
   Taking the inner product of the above equation with $\varphi W$ and contracting $X$ and $W$ and we get the result
   \begin{equation}\label{eqd01}
       S^*(X,Y)=\frac{(\a^2-\b^2)(n-1)-(n+1)c}{4}g(\varphi X,\varphi Y).
   \end{equation}
For a three-dimensional trans-Sasakian space form, we find the result from the above equation\\
\begin{equation}\label{eqd02}
    S^*(X,Y)=-\frac{c}{2}g(\varphi X,\varphi Y).
\end{equation}
Now we say that $M$ is $*$$-\eta-$Einstein.
    From the definition of the $*-$Ricci soliton,
    $$(\mc{L}_V g)(X,Y)+2S^*(X,Y)+2\lambda g(X,Y)=0,$$
    and
    $$g(\nabla_XV,Y)+g(X,\nabla_YV)+2S^*(X,Y)+2\lambda g(X,Y)=0.$$
    We take $V=\xi$ and obtain
    \begin{align*}
        g(\nabla_X\xi,Y)+g(X,\nabla_Y\xi)+2S^*(X,Y)+2\lambda g(X,Y)=0.
    \end{align*}
    Using equation (\ref{eqb06}), we get
    $$\b[g(X,Y)-\eta(X)\eta(Y)]+2S^*(X,Y)+\lambda g(X,Y)=0.$$
    Taking into account (\ref{eqd02}) and contracting $X$ and $Y$ in the above equation, we obtain
    $$\lambda=\frac{1}{3}(c-2\b).$$
    Now we say that the soliton is shrinking, steady and expanding if $c<2\b$, $c=2\b$ and $c>2\b$ respectively.
\end{proof}

\begin{theorem}
     Let M be a three-dimensional trans-Sasakian space form of type $(\a,\b)$, where $\a$ and $\b$ are constants with $\b\ne 0$. If $M$ admits a hyperbolic $*-$Ricci soliton, then $\lambda =\frac{1}{4\b}(3\mu+c-4\b^2)$.
\end{theorem}
\begin{proof}
    From the definition hyperbolic $*-$Ricci soliton,
    $$\mc{L}_\xi(\mc{L}_\xi g)(X,Y)+2\lambda\mc{L}_\xi g(X,Y)+2S^*(X,Y)=2\mu g(X,Y).$$
    Using equation (\ref{eqc06}) and (\ref{eqc07}) in the above equation, then 
    \begin{equation}\label{eqd03}
         2\b(\b+\lambda) (g(X,Y)-\eta(X)\eta(Y))-\frac{c}{2}g(\varphi X,\varphi Y)=\mu g(X,Y).
    \end{equation}
    Replacing $X,Y$ with the basis $\{e_i\}$ of $M$ in equation (\ref{eqd03}), we get this result
    $$\lambda=\frac{1}{4\b}(3\mu+c-4\b^2).$$
    The hyperbolic $*$-Ricci soliton is shrinking, steady, and expanding accordingly as $\lambda$ is negative, zero and positive respectively.
    So we say that the soliton is expanding, steady and shrinking if $c>4\b-3\mu,~c=4\b-3\mu~and~c<4\b-3\mu$ respectively.
\end{proof}

\begin{theorem}
     If a three-dimensional trans-Sasakian space form $(M, g,\varphi,\eta,\xi,\\c,\a,\b)$ with $\b\ne 0$, admits a hyperbolic Ricci soliton, then $M$ is $\eta-$Einstein, $\mu=2(\a^2-\b^2)$ and $\lambda =\frac{(\a^2-\b^2-c)}{2\b}-\b.$ Moreover, the soliton is expanding, steady or shrinking according as $c<\a^2-3\b^2,$ $c=\a^2-3\b^2$ or $c>\a^2-3\b^2$, respectively. 
\end{theorem}
\begin{proof}
    Let $(M,g)$ be a three-dimensional trans-Sasakian space form admitting a hyperbolic Ricci soliton. By the definition of hyperbolic Ricci soliton (\ref{eq004}), replacing $V$ by $\xi$, we have
\begin{equation}\label{eqd04}
    \mc{L}_\xi(\mc{L}_\xi g)(X,Y)+2\lambda\mc{L}_\xi g(X,Y)+2S(X,Y)=2\mu g(X,Y).
\end{equation}   
 Using (\ref{eqc06}) and (\ref{eqc07}) in (\ref{eqd04}), we have  
    \begin{equation}\label{eqd06}
        2\b(\b+\lambda)[g(X,Y)-\eta(X)\eta(Y)]+S(X,Y)=\mu g(X,Y).
    \end{equation}
  Replacing $X$ by $\xi$ in the previous equation, we get
    $$ 2\b(\b+\lambda)[g(\xi,Y)-\eta(\xi)\eta(Y)]+S(\xi,Y)=\mu g(\xi,Y).$$
   Using (\ref{eqb02}) in the above equation, we get
        $$ S(\xi,Y)=\mu \eta(Y).$$
    By equation (\ref{eqb10}), we get
    \begin{equation}\label{mu}
        \mu=2(\a^2-\b^2).
    \end{equation}
    Contracting $X$ and $Y$ in equation (\ref{eqd06}) and using the value of $\mu$ from (\ref{eqb07}), we obtain
    \begin{equation}\label{ju}
        \lambda=\frac{(\a^2-\b^2-c)}{2\b}-\b.
    \end{equation}
The  soliton is expanding if
\begin{equation}
    c<\a^2-3\b^2,
\end{equation}
 steady if
\begin{equation}
      c=\a^2-3\b^2,
\end{equation}
and shrinking if
\begin{equation}
      c>\a^2-3\b^2.
\end{equation}
\end{proof}

\begin{theorem}
    If a three-dimensional trans-Sasakian space form $(M, g,\varphi,\eta,\xi,\\c,\a,\b)$ with $\b\ne 0$ admits hyperbolic conformal $Ricci$ $soliton$ then $M$ is $\eta-$Einstein manifold and $\mu=2(\a^2-\b^2)+\frac{1}{2}(p+\frac{2}{3})$ and $\lambda =\frac{1}{4\b}[3(\mu-\frac{1}{2}(p+\frac{2}{3}))-4\a^2-2c]$.
\end{theorem}

\begin{proof}
    From the definition (\ref{de01}), we rewrite the equation (\ref{eq006}) as
$$ \mc{L}_\xi(\mc{L}_\xi g)(X,Y)+2\lambda\mc{L}_\xi g(X,Y)+2S(X,Y)=2(\mu-\frac{1}{2}(p+\frac{2}{3})) g(X,Y).$$
By inserting (\ref{eqc06}) and (\ref{eqc07}) into the above equation, we deduce 
    \begin{equation}\label{eq0004}
        S(X,Y)+(2\b^2+2\b\lambda)\{g(X,Y)-\eta(X)\eta(Y)\}=[\mu-\frac{1}{2}(p+\frac{2}{3})] g(X,Y),
    \end{equation}
    changing $X$ by $\xi$, we get
    $$\mu=2(\a^2-\b^2)+\frac{1}{2}(p+\frac{2}{3}).$$
    Contracting $X$ and $Y$ in equation (\ref{eq0004}), we obtain this result,
    
    $$\lambda=\frac{1}{4\b}[3(\mu-\frac{1}{2}(p+\frac{2}{3}))-4\a^2-2c].$$
Now we say that the  soliton is
expanding if $$c<\a^2-3\b^2,$$
 steady if $$c=\a^2-3\b^2,$$
and shrinking if $$c>\a^2-3\b^2.$$
\end{proof}
 
\begin{theorem}
    If a three-dimensional trans-Sasakian space form $(M, g)$ admits an Yamabe soliton $(g,\xi,\lambda,\mu)$, then the manifold $(M, g)$ is $\eta-$Einstein and $\mu=r$, where $r$ is the scalar curvature and $\lambda=-2\b$. Furthermore, the soliton is shrinking, steady or expanding according as
    $\b>0,~\b=0,~\b<0$ respectively.
\end{theorem}
\begin{proof}
    Let us consider a three-dimensional trans-Sasakian space form $(M,g)$ that admits a hyperbolic Yamabe soliton $(g,\xi, \lambda,\mu)$. Then, from (\ref{eq007}), we can write

$$\mc{L}_V\mc{L}_V g(X,Y)+\lambda\mc{L}_V g(X,Y)=(\mu-r) g(X,Y),$$
for all $X,Y \in TM$.
We replace $V$ by $\xi$, then we get
\begin{equation}\label{eq2}
    \mc{L}_\xi\mc{L}_\xi g(X,Y)+\lambda\mc{L}_\xi g(X,Y)=(\mu-r) g(X,Y).
\end{equation}
Applying (\ref{eqc06}) and (\ref{eqc07}) to the above equation, we arrive at 
\begin{equation}\label{eq5} 
    4\b^2[g(X,Y)-\eta(X)\eta(Y)]+2\lambda \b [g(X,Y)-\eta(X)\eta(Y)]=(\mu-r) g(X,Y).
\end{equation}
Also, substituting $X=\xi$ into equation (\ref{eq5}), we find that
$$\mu=r.$$
Contracting $X$ and $Y$ in equation (\ref{eq5}), we have
$$ 4\b^2[g(e_i,e_i)-\eta(e_i)\eta(e_i)]+2\lambda \b [g(e_i,e_i)-\eta(e_i)\eta(e_i)]=(\mu-r) g(e_i,e_i).$$
Summing over $i=1,2,3,$ then
$$ 8\b^2+4\lambda \b=3(\mu-r).$$
Putting the value of $\mu=r$, we have
$$\lambda=-2\b.$$
\end{proof}

\section{ Ricci Curvature  On Conformal Trans-Sasakian Space Form }
This section is devoted to studying the nature of Ricci curvature tensor on conformal trans-Sasakian space form.
\begin{theorem}
    If a three-dimensional conformal trans-Sasakian space form $(M,g,\varphi,\eta,\xi,c)$ of type $(\a,\b)$  with $\b\ne 0$ admits a hyperbolic $Ricci~soliton$  then the real scalar $\lambda$ will be $\frac{1}{\b}\{\frac{1}{2}\mu-\b^2-\exp
(f)[\frac{2}{3}(\a^2-\b^2)+\frac{c}{3}-2\Delta f+\frac{5}{12}||\omega^\#||^2]\}.$
\end{theorem}
\begin{proof}
    A three-dimensional smooth manifold $M$ with an almost contact metric structure $(\varphi,\eta,\xi,g)$ is called a conformal trans-Sasakian space form \cite{ARE} if there is a positive smooth function $f:M\longrightarrow\mathbb{R}$ such that 

\begin{equation}
    \Tilde{g}=\exp(f)g,\qquad \Tilde{\xi}=(\exp(-f))^\frac{1}{2}\xi, \qquad \Tilde{\eta}=(\exp(f))^\frac{1}{2}\eta, \qquad \Tilde{\varphi}=\varphi.
\end{equation}

Using the Koszul formula, we derive the following relation between two connections $\Tilde{\nabla}$ and $\nabla$ on $M$ related to the two metrics $\Tilde{g}$ and $g$ respectively.
\begin{equation}
    \Tilde{\nabla}_XY=\nabla_XY+\frac{1}{2}\{\omega(X)Y+\omega(Y)X-g(X,Y)\omega^\#\},
\end{equation}
for all vector fields $X,Y$ on $M$, so that $\omega(X)=X(f)$ and  $\omega^\#$ is a vector field of metrically equivalent to one form of $\omega$, that is, $g(\omega^\#,X)=\omega(X)$. The vector field $\omega^\#=\text{grad}\,f$ is called the Lee vector field \cite{ARE} of the conformal trans-Sasakian manifold $M$. Then we have,
\begin{align}\label{eq00e1}
  \exp(-f)\Tilde{g}(\Tilde{R}(X,Y)Z,W)=g(R(X,Y)Z,W)+\frac{1}{2}\{B(X,Z)g(Y,W)\nonumber\\-B(Y,Z)g(X,W)+B(Y,W)g(X,Z)-B(X,W)g(Y,Z)\}\nonumber\\+\frac{1}{4}||\omega^\#||^2\{g(X,Z)g(Y,W)-g(Y,Z)g(X,W)\},  
\end{align}
for all $X,Y,Z,W$ on $M$, where 

$$B=\nabla\omega-\frac{1}{2}\omega\otimes\omega$$
$$B(X,Y)=(\nabla_X\omega)Y-\frac{1}{2}\omega(X)\omega(Y).$$
Using the condition $\omega(X)=X(f)$, we get
$$B(X,Y)=\nabla_X\nabla_Yf-\frac{1}{2}\omega(X)\omega(Y).$$
Now contracting $X,W$ from equation (\ref{eq00e1}), we get
\begin{equation}\label{eq00e2}
    \exp(-f)\Tilde{S}(X,Y)=S(X,Y)-\frac{1}{2}\{B(X,Y)+(\Delta f-\frac{3}{2}||\omega^\#||^2)g(X,Y)\},
\end{equation}
where $\Delta f=\sum \operatorname{tr} \nabla_X\nabla_Yf$.

Now from the definition of Hyperbolic Ricci soliton on three-dimension conformal trans-Sasakian space form, we get
$$\frac{1}{2}\mc{L}_\xi(\mc{L}_\xi g)(X,Y)+\lambda\mc{L}_\xi g(X,Y)+\Tilde{S}(X,Y)=\mu g(X,Y).$$
Using (\ref{eqc06}), (\ref{eqc07}) and (\ref{eq00e2}) in above equation, we deduce
\begin{align}
    2\b^2g(\varphi^2 X,\varphi^2Y)+2\lambda\b g(\varphi X,\varphi Y)+\exp(f)[S(X,Y)\nonumber\\-\frac{1}{2}\{B(X,Y)+(\Delta f-\frac{3}{2}||\omega^\#||^2)g(X,Y)\}]=\mu g(X,Y).
\end{align}
Contracting $X,Y$ and summing over $i=1,2,3$, we obtain 
$$\lambda=\frac{1}{\b}\{\frac{1}{2}\mu-\b^2-\exp
(f)[\frac{2}{3}(\a^2-\b^2)+\frac{c}{3}-2\Delta f+\frac{5}{12}||\omega^\#||^2]\},$$
where $\b\neq 0$.
\end{proof}

\section{Ricci Curvature on a Submanifold in trans-Sasakian Space Form}
This section is devoted to studying the nature of Ricci curvature tensor on a submanifold of a trans-Sasakian space form.
\begin{theorem}
    If $M$ is a two-dimensional totally umbilical submanifold of  a three-dimensional trans-Sasakian space form $(\widetilde{M},\varphi,\xi,\mu,\gamma,\lambda)$ of type $(\a,\b)$ with $\b\neq 0$ containing Reeb vector field $\xi$ admits hyperbolic Ricci soliton, then $ ||H||^2=\mu-2(\a^2-\b^2)$, where $H$ is the mean curvature of the submanifold and $\lambda=\frac{1}{\b}(\mu+\b^2-||H||^2-2\a^2-c)$.
\end{theorem}
\begin{proof}
    Let $M$ be a two-dimensional submanifold of a three-dimensional Riemannian manifold $\widetilde{M}$ equipped with a Riemannian
metric $\widetilde{g}$. Covariant derivatives and curvatures with respect to $(M,g)$ will follow the same convention as those with respect to the ambient manifold $(\widetilde{M},\widetilde{g})$. We use the inner product notation $(\langle ,\rangle)$ for both the
metric $\widetilde{g}$ of $\widetilde{M}$ and the induced metric $g$ on the submanifold $M$.

The Gauss and Weingarten formulas are given  by (\cite{HTM})
$$\widetilde{\nabla}_XY=\nabla_XY+h(X,Y),$$ and $$\widetilde{\nabla}_XN=-A_NX+\nabla_X^\perp N,$$
for all $ X,Y\in TM$ and $N\in T^\perp M$, where $\widetilde{\nabla}$, $\nabla$ and $\nabla^\perp$ are  the Riemannian, induced Riemannian and induced normal connections in $\widetilde{M}$, $M$ and the normal bundle $T^\perp M$ of $M$ respectively. $h$ is the second fundamental form of $M$ related to the shape operator $A$ given by $g(A_V X,Y) = g(h(X, Y ), V )$. The equation of Gauss is given
by (\cite{ERM})
\begin{align*}
    g(R(X,Y)Z,W)=\widetilde{g}(\widetilde{R}(X,Y)Z,W)+\widetilde{g}(h(X,W),h(Y,Z))\\-\widetilde{g}(h(X,Z),h(Y,W)).\qquad\qquad\qquad\qquad
\end{align*}
Contracting $X,W$, we obtain
\begin{align}\label{eq00e3}
  \sum\limits_{i=1}^{2}g(R(e_i,Y)Z,e_i)=\sum\limits_{i=1}^{2}\widetilde{g}(\widetilde{R}(e_i,Y)Z,e_i)+\sum\limits_{i=1}^{2}\widetilde{g}(h(e_i,e_i),h(Y,Z))\nonumber\\-\sum\limits_{i=1}^{2}\widetilde{g}(h(e_i,Z),h(Y,e_i)),  
\end{align}
for all $Y,Z\in TM$, where $\widetilde{R}$
and $R$ are the curvature tensors of $\widetilde{M}$ and $M$ respectively.
The mean curvature vector $H$ is given by $H =\frac{trace(h)}{2} $, where $2=dim(M)$. Since $M$ is totally umbilical, $h(X,Y)=g(X,Y)H$ $\implies \sum\limits_{i=1}^{2}\widetilde{g}(h(e_i,X),h(Y,e_i))=g(X,Y)||H||^2$ for all $X,Y\in TM$.
 Using equation (\ref{eq00e3}), we get
\begin{equation}
    S(X,Y)=\widetilde{S}(X,Y)+g(X,Y)||H||^2.
\end{equation}
Now we use the above result on the equation of hyperbolic Ricci soliton on a trans-Sasakian space form, we have
$$\frac{1}{2}\mc{L}_\xi(\mc{L}_\xi g)(X,Y)+\lambda(\mc{L}_\xi g)(X,Y)+\widetilde{S}(X,Y)+g(X,Y)||H||^2=\mu g(X,Y)$$
Using equations (\ref{eqc06}), (\ref{eqc07}) in the above equation, we deduce
\begin{equation}\label{eq00e4}
    (2\b^2+2\b\lambda)\{g(X,Y)-\eta(X)\eta(Y)\}+\widetilde{S}(X,Y)+g(X,Y)||H||^2=\mu g(X,Y).
\end{equation}
Replacing $Y$ by $\xi$, we get
$$2(\a^2-\b^2)+||H||^2=\mu,$$
or,
\begin{equation}
    ||H||^2=\mu-2(\a^2-\b^2).
\end{equation}
Contracting $X,Y$ on (\ref{eq00e4}), we obtain
\begin{equation}
    \lambda=\frac{1}{\b}(\mu+\b^2-||H||^2-2\a^2-c).
\end{equation}

\end{proof}

\section{Example}
Let us define a three-dimensional manifold $M=\{(x,y,z)\in\mathbb{R}^3,y\neq 0\}$, where $(x,y,z)$ are the standard co-ordinate in $\mathbb{R}^3$. The vector fields are defined below by

$\quad\quad\quad\quad\quad\quad\quad\quad e_1=\frac{\partial}{\partial x}, \quad\quad
e_2=e^{2x}\frac{\partial}{\partial y}, \quad\quad e_3=e^{2x}\frac{\partial}{\partial z}$\\
are linearly independent at each point of $M$. We define the Riemannian metric $g$ as
\begin{equation}
     g(e_1,e_1)=g(e_2,e_2)=g(e_3,e_3)=1,\quad g(e_1,e_2)=g(e_1,e_3)=g(e_2,e_3)=0.
\end{equation}
Let $\xi=e_1$, then the $1-$form $\eta$ is defined by $\eta(X)=g(X,e_1)$ for arbitrary $X\in\mathfrak{X}(M)$. Then we have the following relations:
\begin{equation}
    \eta(e_1)=1,\quad \eta(e_2)=0, \quad\eta(e_3)=0.
\end{equation}
Let us define $(1,1)$ tensor field $\varphi$ as 
\begin{equation}
    \varphi(e_1)=0, \quad \varphi(e_2)=-e_3, \quad \varphi(e_3)=e_2.
\end{equation}
We can easily check the condition
$$\varphi^2X=-X+\eta(X)\xi, \quad g(\varphi X,\varphi Y)=g(X,Y)-\eta(X)\eta(Y), \quad \forall X,Y\in TM.$$
Hence the structure $(\varphi,\xi,\eta,g)$ defines an almost contact metric structure on $M$. Now from the definition of Lie bracket, we obtain the results
$$[e_1,e_2]=2e_2,\quad[e_2,e_3]=0,\quad[e_1,e_3]=2e_3.$$
Let $\nabla$ be the Levi-Civita connection of the Riemannian metric $g$ defined by the Koszul formula which is given by
\begin{align*}
    2g(\nabla_XY,Z)=Xg(Y,Z)+Yg(X,Z)-Zg(Y,X)-g(X,[Y,Z])\\-g(Y,[X,Z])+g(Z,[X,Y]).\qquad\qquad\qquad\qquad\quad
\end{align*}
Using the above formula, we can easily calculate the following results 
$$\nabla_{e_1}e_1=0\quad\quad\quad \nabla_{e_1}e_2=0\quad\quad\quad\nabla_{e_1}e_3=0$$
$$\nabla_{e_2}e_1=-2e_2 \quad\quad\nabla_{e_2}e_2=2e_1 \quad\quad\nabla_{e_2}e_3=0$$
$$\nabla_{e_3}e_1=-2e_3 \quad\quad\nabla_{e_3}e_2=0 \quad\quad\nabla_{e_3}e_3=2e_1,$$
which satisfies all conditions of the trans-Sasanian manifold. Hence $M$ becomes a trans-Sasakian manifold of type $(0,-2)$. Now, using the above results and the formula of Reimannian curvature tensor $R$, that is $R(X,Y)Z=\nabla_X\nabla_YZ-\nabla_Y\nabla_XZ-\nabla_{[X,Y]}Z$, we get the components of the Riemannian curvature tensor, which are given by \\
$$R(e_1,e_2)e_1=4e_2\quad\quad R(e_2,e_3)e_1=0\quad\quad R(e_1,e_2)e_3=0$$
$$R(e_1,e_3)e_1=4e_3\quad\quad R(e_1,e_3)e_2=0\quad\quad R(e_1,e_3)e_3=-4e_1$$
$$R(e_1,e_2),e_2=-4e_1\quad\quad R(e_2,e_3)e_2=4e_3\quad\quad R(e_2,e_3)e_3=-4e_2$$
and the components of Ricci tensor are given by\\
$$S(e_1,e_1)=-8\quad\quad\quad S(e_2,e_2)=-8\quad\quad\quad S(e_3,e_3)=-8.$$
Hence $M$ is a generalized trans-Sasakian space form with the $\phi-$sectional curvature $c=-4$.
Let we define a vector field $V$ by, $V=\xi$. Then we get\\
$$(\mc{L}_Vg)(e_1,e_1)=0,\quad\quad (\mc{L}_Vg)(e_2,e_2)=-4,\quad\quad (\mc{L}_Vg)(e_3,e_3)=-4,$$
and
$$\mc{L}((\mc{L}_Vg))(e_1,e_1)=0,\quad\quad \mc{L}((\mc{L}_Vg))(e_2,e_2)=16,\quad\quad \mc{L}((\mc{L}_Vg))(e_3,e_3)=16.$$
Using the above results in the equations of hyperbolic Ricci soliton, we find this relation $\lambda=-(1+\frac{3}{8}\mu)$. Now, we putting the values of $\a,~\b~and~ c$ in the equation \ref{mu} and \ref{ju}, we get $\lambda=2$ and $\mu=-8$, which satisfy the relation.

\section{Acknowledgments.} The author Bidhan Mondal thanks the University Grants Commission Junior Research Fellowship (Id No. 231610077944) for their financial assistance.

\end{document}